\documentclass[11pt,a4paper]{amsart}

\usepackage{amsmath,amscd}
\usepackage{amssymb}
\usepackage{amsthm}

\usepackage[ocgcolorlinks, linkcolor=blue]{hyperref}
\usepackage[alphabetic,abbrev,nobysame]{amsrefs} \AtBeginDocument{\def\MR#1{}} 

\usepackage{graphicx}

\usepackage{mathrsfs}

\usepackage{calc}
             {\begin{list}{\arabic{enumi}.}{\usecounter{enumi}%
              \setlength{\labelsep}{0.5em}%
              \settowidth{\labelwidth}{(\arabic{enumi})}%
              \setlength{\leftmargin}{\labelwidth+\labelsep}}}%
             {\end{list}}

\usepackage{comment}

\DeclareRobustCommand{\SkipTocEntry}[5]{}

\usepackage{xcolor}
\definecolor{LOcolor}{RGB}{150,100,0}

\newtheorem{Theorem}{Theorem}[section]

\newtheorem{Lemma}[Theorem]{Lemma}

\newtheorem{Proposition}[Theorem]{Proposition}

\theoremstyle{definition}

\newtheorem*{DefinitionNoNumber}{Definition}

\newtheorem*{ExampleNoNumber}{Example}

\newtheorem{Remark}[Theorem]{Remark}

\numberwithin{equation}{section}

\newcommand{\mR}{\mathbb{R}}                    
\newcommand{\mC}{\mathbb{C}}                    
\newcommand{\abs}[1]{\lvert #1 \rvert}          
\newcommand{\norm}[1]{\lVert #1 \rVert}         

\newcommand{\ol}[1]{\overline{#1}}

\newcommand{\supp}{\mathrm{supp}}

\newcommand{\eps}{\varepsilon}

\newcommand{\p}{\partial}

\newcounter{sidenote}

\begin{document}

\title{Free boundary methods and non-scattering phenomena}

\author{Mikko Salo}
\address{Department of Mathematics and Statistics, University of Jyv\"askyl\"a}
\email{mikko.j.salo@jyu.fi}

\author{Henrik Shahgholian}
\address{Department of Mathematics, KTH Royal Institute of Technology}
\email{henriksh@kth.se}




\begin{abstract}
We study a question arising in inverse scattering theory: given a penetrable obstacle, does there exist an incident wave that does not scatter? We show that every penetrable obstacle with real-analytic boundary admits such an incident wave. At zero frequency, we use quadrature domains to show that there are also obstacles with inward cusps having this property. In the converse direction, under a nonvanishing condition for the incident wave, we show that there is a dichotomy for boundary points of any penetrable obstacle having this property: either the boundary is regular, or the complement of the obstacle has to be very thin near the point. These facts are proved by invoking results from the theory of free boundary problems.
\end{abstract}

\maketitle

\section{Introduction} \label{sec_introduction}

\subsection{Motivation}

In this article we discuss some examples of non-scattering phenomena based on methods from free boundary problems. The connection between these fields was recently pointed out in \cite{CakoniVogelius}, and we invoke further ideas from free boundary problems to obtain stronger results. The methods are relevant both for inverse scattering problems and inverse boundary value problems. We first describe the boundary case and state the main results in that setting, and discuss the scattering case afterwards. All functions will be assumed real valued unless mentioned otherwise.

Let $\Omega \subset \mR^n$ be a bounded domain with smooth boundary, and let $q \in L^{\infty}(\Omega)$ be a potential in $\Omega$. Assuming that $0$ is not a Dirichlet eigenvalue for $\Delta + q$ in $\Omega$, for any $f \in H^{1/2}(\p \Omega)$ there is a unique solution $u \in H^1(\Omega)$ of the Dirichlet problem 
\[
(\Delta + q) u = 0 \text{ in $\Omega$}, \qquad u|_{\p \Omega} = f.
\]
We assume that we can fix a Dirichlet data $f$ and measure the corresponding Neumann data $\p_{\nu} u|_{\p \Omega}$ (interpreted weakly as an element of $H^{-1/2}(\p \Omega)$) on the boundary. This kind of situation arises in diffuse optical tomography \cite{ArridgeSchotland2009}. It is also relevant in electrical impedance tomography, i.e.\ Calder\'on problem \cite{Uhlmann_survey}, where the underlying conductivity equation $\mathrm{div}(\gamma \nabla v) = 0$ can often be reduced to the equation $(\Delta + q) u = 0$ with $q = -\gamma^{-1/2} \Delta(\gamma^{1/2})$ by using the Liouville transformation $v = \gamma^{-1/2} u$.

In inverse boundary value problems of this type, one often assumes the knowledge of the full Dirichlet-to-Neumann map 
\[
\Lambda_q: H^{1/2}(\p \Omega) \to H^{-1/2}(\p \Omega), \ \ \Lambda_q f = \p_{\nu} u|_{\p \Omega}.
\]
This corresponds to an idealized case where we can perform infinitely many measurements. However, in practice only finitely many measurements are possible. Moreover, the idealized problem is formally overdetermined when $n \geq 3$, in the sense that the unknown $q$ is a function of $n$ variables whereas the measurements (Schwartz kernel of $\Lambda_q$) depend on $2n-2$ variables. This suggests that fewer measurements might be sufficient. We are interested in the \emph{single measurement inverse problem}: which properties of $q$ can be determined from the measurement $\p_{\nu} u|_{\p \Omega}$ corresponding to a fixed Dirichlet data $f \in H^{1/2}(\p \Omega)$?

In general, it is not possible to determine a potential $q$ from a single measurement. This is indicated by a heuristic dimension count argument: the measurement $\p_{\nu} u|_{\p \Omega}$ is a function of $n-1$ variables, whereas the unknown potential is a function of $n$ variables. Thus the inverse problem of determining $q$ from a single measurement is formally underdetermined. A related problem is to determine the shape of a \emph{penetrable obstacle} from a single measurement. This corresponds to potentials of the form 
\[
q = h \chi_D, \qquad \ol{D} \subset \Omega ,
\]
where $\chi_D$ is the characteristic function of $D$ (a bounded open set, i.e.\ the obstacle), and $h$ satisfies a nonvanishing condition at $\p D$. We will sometimes assume the following conditions for $D$ and $h$.

\begin{DefinitionNoNumber}
A bounded open set $D \subset \mR^n$ is called a \emph{solid domain} if $D$ and $\mR^n \setminus \ol{D}$ are connected, and $\mathrm{int}(\ol{D}) = D$. We say that $h$ is a \emph{contrast} for $D$ if 
\[
h \in L^{\infty}(\mR^n), \quad \hbox{and } \abs{h} \geq c > 0 \hbox{ a.e.\ in some neighborhood of $\p D$.} 
\]
\end{DefinitionNoNumber}

Since the potential is $q = h \chi_D$, the values of the contrast outside $D$ will not play any role. The inverse problem is to determine the shape of the obstacle, i.e.\ $\p D$, or some properties of $\p D$ from a single measurement $\p_{\nu} u|_{\p \Omega}$. If $\p D$ is (say) a Lipschitz domain, then it is locally the graph of a function of $n-1$ variables and thus the inverse problem is formally well determined.

There are various results in the literature for determining $\p D$ from a single measurement. For a related Calder\'on type problem corresponding to the equation $\mathrm{div}(\gamma \nabla u) = 0$ where $\gamma = 1 + h \chi_D$ and $h$ is a nonzero constant, there are partial results when $D$ has special geometry, such as $D$ being a convex polygon or polyhedron, a ball, or a cylinder. There are also local uniqueness results (if $D$ and $D'$ are close enough in some sense then they can be distinguished by a single measurement) and estimates for the size of $D$. 
See \cite{Alessandrini1999} for a survey of classical results, and \cite{LiuTsou2020} and references therein for more recent results. The results in \cites{AlessandriniIsakov, AthanasopoulosCaffarelliSalsa} are of particular interest to us: they invoke methods from free boundary problems to show that if $D$ is e.g.\ a Lipschitz domain and it is not determined by a single measurement, then part of $\p D$ is necessarily real-analytic. We refer to \cites{KLS_jump, KLS_nodal, FKS_vectorial} for recent related results.


It turns out that such results are closely connected to a certain non-scattering phenomenon in inverse scattering theory. These problems involve a fixed frequency $\lambda \geq 0$. Given a bounded open set $\Omega \subset \mR^n$, we ask if there exist nontrivial solutions of 
\begin{equation} \label{eq1}
(\Delta + \lambda^2 + q) u_q = 0 \text{ in $\Omega$}, \qquad (\Delta + \lambda^2) u_0 = 0 \text{ in $\Omega$}
\end{equation}
that satisfy 
\begin{equation} \label{eq2}
u_q, u_0 \in H^1(\Omega), \qquad u_q-u_0 \in H^2_0(\Omega).
\end{equation}
This problem is rather similar to the interior transmission problem (see \cites{CakoniHaddar, CakoniColtonHaddar}) but note that we require $u_q, u_0$ to be $H^1$ instead of $L^2$. The above problem can in fact be considered as a matching problem as in \cite{FS_matching}. One could also accommodate the condition $u_q-u_0 - g \in H^2_0(\Omega)$ for some smooth enough $g$, which would be close to the inhomogeneous interior transmission problem.

If the boundary of $\Omega$ is smooth enough, the condition \eqref{eq2} can be written as 
\[
u_q|_{\p \Omega} = u_0|_{\p \Omega}, \qquad \p_{\nu} u_q|_{\p \Omega} = \p_{\nu} u_0|_{\p \Omega}.
\]
Thus, if one fixes the Dirichlet data $f = u_0|_{\p \Omega} = u_q|_{\p \Omega}$, this would mean that the measurement $\p_{\nu} u_q|_{\p \Omega}$ corresponding to $q$ is identical to the measurement $\p_{\nu} u_0|_{\p \Omega}$ for the zero potential. Thus the potential $q$ is invisible for this particular measurement and looks like empty space. In the terminology of scattering theory, if this happens we say that "the incident wave $u_0$ does not scatter".

We specialize the above question to the case of an obstacle $D$. The following is the main question studied in this article:

\begin{quote}
Given a bounded open set $D \subset \mR^n$ with $\ol{D} \subset \Omega$,  is there a contrast $h$ for $D$ such that there exist nontrivial solutions $u_q$ and $u_0$ with $q = h \chi_D$ satisfying \eqref{eq1}--\eqref{eq2}?
\end{quote}

If the answer is positive, then there is some contrast $h$ for $D$ that admits an incident wave that does not scatter (thus $D$ will be invisible with respect to this measurement). On the other hand, if the answer is negative, then the obstacle $D$ scatters every incident wave nontrivially.

\subsection{Main results}

There are various results stating that if $\p D$ is piecewise smooth and has a corner singularity, then every incident wave will scatter nontrivially. We will give precise references in Section \ref{subseq_scattering}. On the other hand, there seem to be few examples in the literature of penetrable obstacles admitting incident waves that do not scatter. Balls have this property \cite[Sections 10.3 and 8.4]{ColtonKress}, and \cite{GellRedmanHassell} gives examples of potentials $q \in C^{\infty}_c(\mR^n)$ having this property whose supports are unions of balls. See \cite{VogeliusXiao} for some related results.

Our first result states that any obstacle with real-analytic boundary admits incident waves that do not scatter:

\begin{Theorem} \label{thm_fb_example_analytic_intro}
Let $\Omega \subset \mR^n$ be a bounded open set, let $\lambda \geq 0$, and let $D \subset \mR^n$ be a bounded open set with real-analytic boundary such that $\ol{D} \subset \Omega$ and $\mR^n \setminus \ol{D}$ is connected. Suppose that $\lambda$ is not a Dirichlet eigenvalue for $-\Delta$ in $D$. Then there is a contrast for $D$ that admits an incident wave that does not scatter.
\end{Theorem}

While corner singularities typically scatter every incident wave, we show that at least for $\lambda = 0$ there also exist obstacles with inward cusp singularities admitting incident waves that do not scatter. We say that a connected open set $D \subset \mR^n$ is a \emph{quadrature domain} (for harmonic functions) if there is a compactly supported distribution $\mu$ in $\mR^n$ with $\supp(\mu) \subset D$ such that 
\begin{equation} \label{quadrature_condition}
\int_D H \,dx = \int_{D} H \,d\mu
\end{equation}
whenever $H \in L^1(D)$ is harmonic in $D$.  A basic example is a ball $B(a,r) \subset \mR^n$ with $\mu = \abs{B_r} \delta_a$, so that \eqref{quadrature_condition} holds by the mean value theorem. There exist many examples of quadrature domains, and their boundaries can exhibit inward cusps (see \cite[Chapter 14]{Davis} or \cite{Sakai1982} for examples). One example is the cardioid domain $D = \{ w + \frac{1}{2} w^2 \,:\, w \in \mathbb{D} \} \subset \mR^2$ that has an inward cusp, see \cite[Figure 0.1]{PetrosyanShahgholianUraltseva} (though note that $\p D$ is the image of $S^1$ by an analytic map).

\begin{Theorem} \label{thm_fb_example_quadrature_intro}
Let $\Omega \subset \mR^n$ be a bounded open set, let $\lambda=0$, and let $D \subset \mR^n$ be a quadrature domain such that $\ol{D} \subset \Omega$. Then there is a contrast for $D$ that admits an incident wave that does not scatter.
\end{Theorem}


As non-scattering incident waves in Theorems \ref{thm_fb_example_analytic_intro} and \ref{thm_fb_example_quadrature_intro}, one can choose any solution of 
\begin{equation} \label{heq}
(\Delta + \lambda^2) u_0 = 0 \text{ in $\Omega$}
\end{equation}
such that $u_0$ is positive on $\p D$. A nonvanishing condition for $u_0$ on $\p D$ will be important for many results in this article (with the exception of Theorem  \ref{Thm:FB-reg1}), and it is of interest to determine if such solutions $u_0$ exist. 
They always do when $\lambda = 0$ (take $u_0 \equiv 1$) or when $u_0$ is allowed to be complex valued (take $u_0 = e^{i \lambda x_1}$). However, by Lemma \ref{lemma_helmholtz_local_zeros} any real valued solution of \eqref{heq} has a zero in any ball of radius $\geq c_n/\lambda$ and the nonvanishing condition on $\p D$ is nontrivial in this case. In fact if $\lambda$ is a Dirichlet eigenvalue of $-\Delta$ in $D$ then solutions $u_0$ satisfying the nonvanishing condition may not exist (see Remark \ref{remark_nonvanishing_false}). On the positive side we will show the following result.

\begin{Theorem} \label{thm_helmholtz_nonvanishing}
Let $D$ be a bounded $C^1$ domain, or Lipschitz domain when $n=2,3$, with $\mR^n \setminus \ol{D}$ connected. Suppose that $\lambda > 0$ is not a Dirichlet eigenvalue of $-\Delta$ in $D$. Then there is a real valued solution $u_0$ of $(\Delta + \lambda^2) u_0 = 0$ in $\mR^n$ such that $u_0$ is positive on $\p D$.
\end{Theorem}

Theorems \ref{thm_fb_example_analytic_intro} and \ref{thm_fb_example_quadrature_intro} are not difficult to prove, and they are analogous to certain facts in the theory of free boundary problems. As mentioned above the connection between single measurement inverse problems and free boundary methods is classical in the Calder\'on problem \cites{AlessandriniIsakov, AthanasopoulosCaffarelliSalsa}. Curiously, it seems that for non-scattering phenomena this connection was only pointed out very recently in \cite{CakoniVogelius}. The main point is the following: if an obstacle $D$ admits an incident wave that does not scatter, then $\p D$ can be understood as a free boundary in a certain obstacle type problem. This observation was used in \cite{CakoniVogelius} to show that if $D$ has Lipschitz boundary and the incident wave $u_0$ is nonvanishing on $\p D$, then necessarily $\p D$ must be real-analytic (resp.\ $C^{k+1,\alpha}$) if the contrast is real-analytic (resp.\ $C^{k,\alpha}$).

We prove a corresponding result where the a priori assumption that $D$ has Lipschitz boundary is removed. However, as indicated by Theorem \ref{thm_fb_example_quadrature_intro}, one must then allow for the possibility that $D$ has inward cusps. We first need to introduce the concept of minimal diameter. 

For any set $K$, we define  $\operatorname{MD} (K )$ to be the minimal diameter of $K$, i.e., 
 the infimum of distances between pairs
of parallel planes such that $K$ is contained in the strip determined by
the planes. For any ball $B(z,r)$ we also define the thickness function 
\[
\delta_r (K,z):=\frac{\operatorname{MD}(K\cap B(z,r))}{r}.
\]
To illustrate this notion, note that if $D \subset \mR^n$ is a bounded Lipschitz domain, then there are $c, r_0 > 0$ such that 
\[
\delta_r(\mR^n \setminus D, x^0) \geq c \text{ whenever $x^0 \in \p D$ and $0 < r < r_0$.}
\]
On the other hand, if $D \cap B(0,1) = \{ x_n < \abs{x'}^{1/\gamma} \} \cap B(0,1)$ where $\gamma > 1$, so that $D$ has an inward cusp at $0$, one can check that 
\[
\delta_r(\mR^n \setminus D, 0) \leq C r^{\gamma-1}, \qquad 0 < r < 1.
\]

We first state the following result showing that if $D$ admits an incident wave $u_0$ that does not scatter, and if both the contrast $h$ and $u_0$ are nonvanishing at a point $x^0 \in \p D$, then there are two possibilities: either $D$ is regular near $x^0$, or the complement of $D$ is thin near $x^0$.

\begin{Theorem}\label{Thm:FB-reg-dichotomy}
Let $\Omega \subset \mR^n$ be a bounded open set, and suppose that $u_q, u_0 \in H^1(\Omega)$ satisfy \eqref{eq1}--\eqref{eq2} where $q = h \chi_D$ for some solid domain $D$ with $\ol{D} \subset \Omega$. Assume that $h$ is Dini continuous (i.e.\ $h$ has a modulus of continuity $\omega$ with $\int_0^1 \omega(r) \,d(\log r) < \infty$). For any $x^0 \in \p D$ such that 
\[
h(x^0) u_0(x^0) \neq 0,
\]
one of the following conditions holds:
\begin{enumerate}
\item[(a)] 
$\limsup_{r \to 0} \delta_r(\Omega \setminus D, x^0) > 0$, and $D$ is locally a $C^1$ domain near $x^0$; or
\item[(b)] 
$\Omega \setminus D$ is thin near $x^0$ in the sense that $\lim_{r \to 0} \delta_r(\Omega \setminus D, x^0) = 0$.
\end{enumerate}
If $h$ is additionally assumed to be Lipschitz (resp.\ $C^{k,\alpha}$ where $k \geq 1$ and $0 < \alpha < 1$, or real-analytic) near $x^0$ and if (a) holds, then $D$ is locally a $C^{1,\alpha}$ (resp.\ $C^{k+1,\alpha}$, or real-analytic) domain near $x^0$.
\end{Theorem}

The case where $h$ is Dini continuous is a consequence of the following more precise result from \cite[Theorem 1.3]{AnderssonLindgrenShahgholian}:

\begin{Theorem}\label{Thm:FB-reg}
Retain the hypotheses of Theorem \ref{Thm:FB-reg-dichotomy}, with  $h$ being a Dini continuous function. Suppose for some $r_1>0 $,  $x^0 \in \partial D$   we have 
\begin{equation}\label{eq:u0-1}
h(x) u_0(x) \neq 0 \qquad    \hbox{for } \quad  x \in  B(x^0,r_1) .
\end{equation}
Then there exists  a modulus of continuity $\sigma (r)$, and a universal constant $\tau >0$, 
such that 
if for some $0 < r_0 < r_1$, we have $\sigma (r_0 ) < \delta_{r_0} (\{u_0=u_q\},x^0) $ then 
$\partial D \cap B(x^0, \tau r_0)$ is a $C^1$-graph.
\end{Theorem}



If $h$ is Lipschitz, or if $h$ is $C^{1,1}$ and $u_0$ vanishes on $\p D$ but $\nabla u_0$ is nonvanishing, we also have the following result.

\begin{Theorem}\label{Thm:FB-reg1}

Retain the hypotheses of Theorem \ref{Thm:FB-reg-dichotomy}, and suppose that either of the following conditions is satisfied:
\begin{enumerate}
\item $h\in C^{0,1}(B(x^0,r))$, 
\item   $h \in C^{1,1} (B(x^0,r))$  and   condition \eqref{eq:u0-1} is replaced by
\begin{equation}\label{eq:u0-2}
 \abs{h} > 0 \quad \hbox{and} \quad u_0= 0 \quad  \hbox{and}  \quad |\nabla u_0 | >0  \quad \hbox{on } \partial D \cap B(x^0,r).
\end{equation}
\end{enumerate}
 Then, there exists $r_0 >0 $ such that one of the following holds:
\begin{enumerate}
\item[(a)]  $\partial D$ is a  $C^{1,\alpha}$  graph  in  $B(x^0,r_0)$.
\item[(b)]  In a translated and rotated system of coordinates 
\[ \partial D \cap B(x^0,r_0)\subset \{ x: \  |x_1| < k (x') \} , \]
where $x' = (x_2, \cdots, x_{n})$, and $k(x') \geq 0 $ is a $C^1$-function with $k(0') =0$.
\end{enumerate}

\end{Theorem}

%
%
%
%
%
%
%

It is noteworthy that Theorem \ref{Thm:FB-reg1}  can be "calibrated" to the case where $u_0$ vanishes to a fixed higher order on some part of $\partial D$, by asking higher order regularity for the right hand side. Notwithstanding this, it remains a tantalizing problem when the higher order vanishing of $u_0$ takes place on isolated points of   $\partial D$. This remains to be studied in the future. See
 \cite{Yeressian} for partial results  in this direction.

\subsection{Connection to free boundary problems}\label{sec:fb-connection}

We now describe more precisely how the existence of an incident wave that does not scatter leads to a free boundary problem. Let $\Omega \subset \mR^n$ be a bounded open set, and suppose that $u_q, u_0 \in H^1(\Omega)$ satisfy \eqref{eq1}--\eqref{eq2} where $q = h \chi_D$ for some solid domain $D$ with $\ol{D} \subset \Omega$ and for $h \in C(\mR^n)$. Then $u_0$ is real-analytic, and also $u_q \in W^{2,p}_{\mathrm{loc}}(\Omega)$ for any $p < \infty$ by elliptic regularity. We write $u := u_q - u_0 \in H^2_0(\Omega)$ and extend $u$ by zero to $\mR^n$. Then $u \in H^2(\mR^n)$ satisfies 
\begin{equation} \label{fbp_eq1}
(\Delta + \lambda^2) u = f_0 \chi_D \text{ in $\mR^n$}
\end{equation}
where $f_0 = -h u_q$ near $\ol{D}$. Note also that since $u$ solves $(\Delta + \lambda^2) u = 0$ in $\mR^n \setminus \ol{D}$ and $u|_{\mR^n \setminus \ol{\Omega}} = 0$, unique continuation implies that $u|_{\mR^n \setminus \ol{D}} = 0$ using that $\mR^n \setminus \ol{D}$ is connected.

Suppose that 
\begin{equation} \label{fbp_hu_nonvanishing}
h(z) u_0(z) \neq 0 \text{ at some $z \in \p D$}.
\end{equation}
Since $u_q = u_0$ outside $\ol{D}$ and $u_q$ is continuous, we also have $f_0(z) \neq 0$.  We claim that 
\begin{equation} \label{fbp_eq2}
\ol{D} \cap B(z,r) = \supp(u) \cap B(z,r)
\end{equation}
whenever $r > 0$ is such that $\abs{h u_q} > 0$ in $B(z,r)$. Since $u|_{\mR^n \setminus \ol{D}} = 0$, one always has $\supp(u) \cap B(z,r) \subset \ol{D} \cap B(z,r)$. Conversely, let $x \in \ol{D} \cap B(z,r)$. If $x \notin \supp(u)$ then $u = 0$ near $x$, which by \eqref{fbp_eq1} implies that $f_0 = 0$ near $x$ which is impossible since $\abs{f_0} > 0$ in $B(z,r)$. This proves \eqref{fbp_eq2}.

Since $D$ is a solid domain, it follows from \eqref{fbp_eq2} that 
\[
D \cap B(z,r) = \mathrm{int}(\supp(u)) \cap B(z,r).
\]
Thus \eqref{fbp_eq1} implies that 
\begin{equation} \label{fbp_eq3}
(\Delta + \lambda^2) u = f_0 \chi_{\mathrm{int}(\supp(u))} \text{ in $B(z,r)$}, \qquad z \in \p(\mathrm{int}(\supp(u))).
\end{equation}
Writing $f = f_0 -\lambda^2 u$, so that $f(z) = f_0(z) \neq 0$, we may further write this as 
\begin{equation} \label{fbp_eq4}
\Delta u = f \chi_{\mathrm{int}(\supp(u))} \text{ in $B(z,r)$}, \qquad z \in \p(\mathrm{int}(\supp(u))).
\end{equation}
The last equation only involves $u$ and not $D$. Thus we have reduced our original problem to an \emph{obstacle problem} in free boundary theory, where locally near $z$ the obstacle is the set $\mathrm{int}(\supp(u))$ and its boundary $\p (\mathrm{int}(\supp(u)))$ can be understood as a free boundary. In the free boundary literature it is more customary to work with equations like  
\begin{equation} \label{fbp_eq5}
\Delta u = f \chi_{\{ u \neq 0 \}} \text{ in $B(z,r)$}, \qquad z \in \p(\{ u = \abs{\nabla u} = 0 \}).
\end{equation}
However, the methods for proving regularity of the free boundary in \eqref{fbp_eq5} also apply to \eqref{fbp_eq4}. See \cite{PetrosyanShahgholianUraltseva} for even more general equations when $f$ is assumed Lipschitz, or \cite{AnderssonLindgrenShahgholian} for $f$ Dini.

The standard obstacle problem corresponds to \eqref{fbp_eq5} for solutions $u \geq 0$ (and $f \equiv 1$ in the most classical case). In our case $u = u_q - u_0$ and it is not possible to assume that $u$ is nonnegative. This means that \eqref{fbp_eq5} corresponds to a \emph{no-sign obstacle problem}. The no-sign assumption on $u$ makes the analysis of this problem extremely hard and one has to resort to advanced tools such as monotonicity formulas along with strong geometric analysis. On the other hand if $\partial D$ is Lipschitz close to $z $ then it follows by standard free boundary techniques that $u$ has a sign in a vicinity of  $z$; see the beginning of the proof of Theorem 1.3 in \cite{AnderssonLindgrenShahgholian}. Thus the case of Lipschitz domains falls back to the regularity theory for the standard obstacle problem, which is rather classical \cite{Caffarelli1977}.

There is by now also a well developed theory for no-sign obstacle problems when $f$ is nonvanishing at the point $z$ of interest, i.e.\ when \eqref{fbp_hu_nonvanishing} holds. We refer to \cite{PetrosyanShahgholianUraltseva} for an account of this theory. After the reduction to \eqref{fbp_eq4}, Theorems \ref{Thm:FB-reg-dichotomy}--\ref{Thm:FB-reg1} follow rather directly from this theory. One can also accommodate the possibility that $f$ vanishes to some fixed order at each point of $\p D \cap B(z,r)$ (see Theorem \ref{Thm:FB-reg1} for an example result). If $f$, or $u_0$, only vanishes at $z$ (or in a set of  dimension $\leq  n-2$), then the problem becomes non-standard and is more or less untouched in the free boundary literature.

\subsection{Inverse scattering} \label{subseq_scattering}

Finally, we discuss the case of inverse scattering problems where the bounded domain $\Omega$ is replaced by $\mR^n$. In this subsection we allow functions to be complex valued.

Let $\lambda > 0$ be a fixed frequency, and let $u_0$ be a solution of $(\Delta + \lambda^2) u_0 = 0$ in $\mR^n$. We consider $u_0$ as an incident wave that is used to probe a medium whose scattering properties are described by a compactly supported potential $q \in L^{\infty}(\mR^n)$. The incident wave $u_0$ induces a total wave $u_q = u_0 + v$ that solves 
\[
(\Delta + \lambda^2 + q)u_q = 0 \text{ in $\mR^n$}.
\]
The solution is unique if we require that the scattered wave $v$ is \emph{outgoing} in the sense that 
\[
v = (-\Delta - (\lambda+i0)^2 - q)^{-1}(-q u_0)
\]
where $(-\Delta - (\lambda+i0)^2 - q)^{-1}$ is the outgoing resolvent. Since $q u_0$ is compactly supported, $v$ has the asymptotics 
\[
v(r\theta) = e^{i\lambda r} r^{-\frac{n-1}{2}} u_q^{\infty}(\theta) + o(r^{-\frac{n-1}{2}}) \qquad \text{as $r \to \infty$}
\]
where $\theta \in S^{n-1}$. The function $u_q^{\infty}$ on $S^{n-1}$ is called the \emph{far field pattern} corresponding to incident wave $u_0$, and it can be measured from the knowledge of $u_q$ as $\abs{x} \to \infty$. We refer to \cites{ColtonKress, Yafaev} for these basic facts.

A commonly used class of incident waves is given by the Herglotz waves, which are solutions of $(\Delta + \lambda^2) u_0 = 0$ having the form 
\[
u_0(x) = \int_{S^{n-1}} e^{i\lambda x \cdot \omega} f(\omega) \,d\omega, \qquad f \in L^2(S^{n-1}).
\]
A scattering analogue of the Dirichlet-to-Neumann map is given by the far field operator 
\[
A_q(\lambda): L^2(S^{n-1}) \to L^2(S^{n-1}), \ \ A_q(\lambda)f = u_q^{\infty}.
\]
A standard fixed frequency inverse problem is to determine $q$ from the knowledge of $A_q(\lambda)$, which corresponds to infinitely many measurements. However, we wish consider the \emph{single measurement problem} in (fixed frequency) inverse scattering: determine some properties of $q$ from knowledge of the far field pattern $u_q^{\infty}$ corresponding to a fixed incident wave $u_0$. If $u_q^{\infty} \equiv 0$ we say that the incident wave $u_0$ does not scatter. Again, in order to obtain a formally well determined problem we consider the case of \emph{penetrable obstacles}, so that 
\[
q = h \chi_D
\]
where $D \subset \mR^n$ is a bounded open set (the obstacle) and $h$ is a contrast for $D$.

In the imaging community it has been understood for a long time that if $\p D$ has corner singularities, one often has strong scattering effects. A rigorous analysis of this phenomenon was initiated in the important work \cite{BlastenPaivarintaSylvester} which showed that if part of $\p D$ is part of a cube, then every incident wave scatters nontrivially for every frequency $\lambda > 0$. In two dimensions this was extended to sectors with angle $< 90^{\circ}$ and single measurement results in \cites{PaivarintaSaloVesalainen, HuSaloVesalainen}. The analysis was based on studying Laplace transforms of characteristic functions of cones via complex geometrical optics solutions as in the Calder\'on problem. There are several related results including quantitative bounds even when corners are replaced by high curvature points, see  \cites{Blasten2018, BlastenLiu1, BlastenLiu2, BlastenLiu3, BlastenVesalainen} and the survey \cite{Liu_survey} (which also discusses results for electromagnetic and elastic scattering). Another important approach to non-scattering problems, introduced in \cite{ElschnerHu2} (see \cites{ElschnerHu1, LiHuYang} for related work), is based on the theory of boundary value problems in corner domains and can be used to produce similar results even for curvilinear polyhedra or when $h$ vanishes to  finite order at $\p D$. We mention that these results related to corner singularities are most complete for $n=2$, and even when $n=3$ they become more limited and mostly apply to edge or circular cone singularities. Finally, the results in \cite{CakoniVogelius}, already discussed before, show regularity of the free boundary if the obstacle is a Lipschitz domain and the incident wave is nonvanishing on its boundary.

In \cite{BlastenPaivarintaSylvester} a frequency $\lambda > 0$ was called a \emph{non-scattering wavenumber} if there is some incident wave $u_0$ that does not scatter. The results mentioned above show that if $\p D$ has corner singularities, then there are no non-scattering wavenumbers (i.e.\ every incident wave scatters nontrivially independent of the frequency). This notion is connected with \emph{interior transmission eigenvalues} (see \cites{CakoniHaddar, CakoniColtonHaddar}) in the sense that a non-scattering wavenumber is also an interior transmission eigenvalue. The converse is not true: for a given potential there are typically infinitely many interior transmission eigenvalues whereas the set of non-scattering frequencies may be empty.

The free boundary approach described above extends directly to the scattering case. If $u_0$ is an incident wave solving $(\Delta + \lambda^2) u_0 = 0$ in $\mR^n$, let $u_q$ be the outgoing solution of $(\Delta + \lambda^2 + q) u_q = 0$ in $\mR^n$ defined above where $q = h \chi_D$ and $D$ is a bounded solid domain. Suppose that $u_0$ does not scatter, i.e.\ $u_q^{\infty} \equiv 0$. Then the Rellich uniqueness theorem and unique continuation imply that $u_q = u_0$ in $\mR^n \setminus \ol{D}$. We may thus take $\Omega$ to be some large ball containing $\ol{D}$, and we are back in the situation of \eqref{eq1}--\eqref{eq2}. Moreover, if $h$ is assumed to be real valued and if for some $z \in \p D$ one has $h(z) u_0(z) \neq 0$, then $\mathrm{Re}(u_0)(z) \neq 0$ or $\mathrm{Im}(u_0)(z) \neq 0$. Thus by taking real or imaginary parts of the solutions one can reduce to a situation where the functions involved are real valued.

If the obstacle has real-analytic boundary, combining Theorem \ref{thm_fb_example_analytic} (with $\Omega = \mR^n$) and Theorem \ref{thm_helmholtz_nonvanishing} leads to an analogue of Theorem \ref{thm_fb_example_analytic_intro} in the scattering setting. This provides examples of real valued contrasts and incident waves that do not scatter.

\subsection*{Structure of the article}

In Section \ref{sec_fb_examples} we will prove Theorems \ref{thm_fb_example_analytic_intro} and \ref{thm_fb_example_quadrature_intro} by using a simple extension argument, the Cauchy-Kowalevski theorem and the defining property of quadrature domains. In Section \ref{sec_helmholtz} we discuss Helmholtz solutions and prove Theorem \ref{thm_helmholtz_nonvanishing}, which follows by combining a result in $D$ with a Runge approximation argument. In Section \ref{sec_fb} we discuss how Theorems \ref{Thm:FB-reg-dichotomy}--\ref{Thm:FB-reg1} follow from arguments in the theory of free boundaries. 

\subsection*{Acknowledgements}

M.S.\ was supported by the Academy of Finland (Finnish Centre of Excellence in Inverse Modelling and Imaging, grant numbers 312121 and 309963) and by the European Research Council under Horizon 2020 (ERC CoG 770924). 
H. Sh. was supported in part by Swedish Research Coucil.

\section{Examples of free boundaries} \label{sec_fb_examples}

In this section we show that any real-analytic boundary, or the boundary of any quadrature domain when $\lambda = 0$, can be realized as a free boundary. The following results are more precise versions of Theorems \ref{thm_fb_example_analytic_intro} and \ref{thm_fb_example_quadrature_intro}, since they also give information on the kinds of incident waves and contrasts for which one has no scattering. Note that Theorem \ref{thm_fb_example_analytic_intro} follows by combining Theorems \ref{thm_fb_example_analytic} and \ref{thm_helmholtz_nonvanishing}, and Theorem \ref{thm_fb_example_quadrature_intro} follows from Theorem \ref{thm_fb_example_quadrature} by taking $u_0 \equiv 1$.

\begin{Theorem} \label{thm_fb_example_analytic}
Let $\lambda \geq 0$, let $D \subset \mR^n$ be a bounded open set with real-analytic boundary, let $\Omega \subset \mR^n$ be an open set with $\ol{D} \subset \Omega$, and let $u_0$ be any solution of $(\Delta + \lambda^2) u_0 = 0$ in $\Omega$ that is positive on $\p D$. Let also $h_0$ be a real-analytic function near $\p D$. Then there is $h \in C^{\infty}(\ol{\Omega})$ with $h=h_0$ near $\p D$ so that with the choice $q = h \chi_D$, the equation 
\[
(\Delta + \lambda^2 + q) u_q = 0 \text{ in $\Omega$}
\]
has a solution $u_q \in H^2(\Omega)$ satisfying $u_q = u_0$ in $\Omega \setminus \ol{D}$.
\end{Theorem}

\begin{Theorem} \label{thm_fb_example_quadrature}
Let $D \subset \mR^n$ be a bounded quadrature domain, let $\Omega \subset \mR^n$ be an open set with $\ol{D} \subset \Omega$, and let $u_0$ be any solution of $\Delta u_0 = 0$ in $\Omega$ that is positive on $\p D$. Then there is $h \in L^{\infty}(\Omega)$ with $\abs{h} \geq c > 0$ near $\p D$ so that with the choice $q = h \chi_D$, the equation 
\[
(\Delta + q) u_q = 0 \text{ in $\Omega$}
\]
has a solution $u_q \in H^2(\Omega)$ satisfying $u_q = u_0$ in $\Omega \setminus \ol{D}$.
\end{Theorem}

The proofs of Theorems \ref{thm_fb_example_analytic} and \ref{thm_fb_example_quadrature} involve the following simple result. It begins with a solution of $(\Delta + \lambda^2) u_0 = 0$ in $\Omega$ that is positive on $\p D$ and with a local solution $v_0$, for some potential $h_0$, that extends $u_0$ slightly inside $D$. The result gives a solution $v$ that extends $u_0$ all the way into $D$ and corresponds to some potential $h \chi_D$ where $h$ extends $h_0$ into $D$. The point is that one first chooses a suitable extension $v$ of $v_0$, and then constructs the potential $h$ depending on $v$.

\begin{Lemma} \label{lemma_local_global}
Let $\Omega \subset \mR^n$ be open, let $D \subset \mR^n$ be a bounded open set with $\ol{D} \subset \Omega$, let $\lambda \geq 0$, and assume that $(\Delta + \lambda^2) u_0 = 0$ in $\Omega$ with $u_0$ positive on $\p D$. Suppose that $U$ is a neighborhood of $\p D$ in $\Omega$ and that $h_0 \in L^{\infty}(U)$ and $v_0 \in C^{1,1}(U)$ satisfy 
\[
(\Delta + \lambda^2 + h_0) v_0 = 0 \text{ in $D \cap U$}, \qquad v_0|_{U \setminus \ol{D}} = u_0|_{U \setminus \ol{D}}.
\]
Then there are $h \in L^{\infty}(\Omega)$ and $v \in C^{1,1}(\Omega)$, with $h = h_0$ and $v = v_0$ near $\p D$, so that 
\[
(\Delta + \lambda^2 + h \chi_D) v = 0 \text{ in $\Omega$}, \qquad v|_{\Omega \setminus \ol{D}} = u_0|_{\Omega \setminus \ol{D}}.
\]
If additionally $h_0 \in C^{\infty}(U)$ and $v_0|_{\ol{D} \cap U} \in C^{\infty}(\ol{D} \cap U)$, then $h \in C^{\infty}(\ol{\Omega})$.
\end{Lemma}
\begin{proof}
Note that $v_0$ is positive in some neighborhood $U_1 \subset U$ of $\p D$, since $v_0 \in C^{1,1}(U)$ and $v_0|_{U \setminus \ol{D}} = u_0|_{U \setminus \ol{D}}$ and $u_0$ is positive on $\p D$. Let $\psi \in C^{\infty}_c(U_1)$ satisfy $0 \leq \psi \leq 1$ and $\psi = 1$ near $\p D$, and define 
\[
v = \left\{ \begin{array}{cl} v_0 \psi + (1-\psi) &\text{ in $\ol{D}$}, \\[5pt] 
 u_0 &\text{ in $\Omega \setminus \ol{D}$}. \end{array} \right.
\]
Then $v \in C^{1,1}(\Omega)$ is positive near $\ol{D}$ and satisfies $v = v_0$ near $\p D$. One can now define a function $h \in L^{\infty}(\Omega)$ by 
\[
h = \left\{ \begin{array}{cl} -\frac{(\Delta + \lambda^2)v}{v} &\text{ in $\ol{D}$}, \\[5pt] 
 \psi h_0 &\text{ in $\Omega \setminus \ol{D}$}. \end{array} \right.
\]
The functions $h$ and $v$ will have the required properties.

If additionally $h_0 \in C^{\infty}(U)$ and  $v_0|_{\ol{D} \cap U} \in C^{\infty}(\ol{D} \cap U)$, it follows that $v|_{\ol{D}} \in C^{\infty}(\ol{D})$ and $h|_{\ol{D}} \in C^{\infty}(\ol{D})$. Since $h = h_0$ near $\p D$ we have $h \in C^{\infty}(\ol{\Omega})$.
\end{proof}

By Lemma \ref{lemma_local_global}, the proofs of Theorems \ref{thm_fb_example_analytic} and \ref{thm_fb_example_quadrature} are reduced to finding a local solution $v_0$ that extends $u_0$ a little bit inside $D$. In the real-analytic case this can be done by solving a Cauchy problem using the Cauchy-Kowalevski theorem.

\begin{proof}[Proof of Theorem \ref{thm_fb_example_analytic}]
Note that $u_0$ is real-analytic in $\Omega$. For any $x \in \p D$, we may use the Cauchy-Kowalevski theorem to find a real-analytic solution of 
\[
(\Delta + \lambda^2 + h_0) v_x = 0 \text{ in $U_x$}, \qquad v_x|_{U_x \cap \p D} = u_0|_{U_x \cap \p D}, \ \ \p_{\nu} v_x|_{U_x \cap \p D} = \p_{\nu} u_0|_{U_x \cap \p D}
\]
where $U_x$ is an open set of the form $\{ z + t\nu(z) \,:\, z \in V_x, \ \abs{t} < \eps_x \}$ with $V_x$ a neighborhood of $x$ in $\p D$ and $\eps_x > 0$. Any two solutions $v_x$ and $v_y$ agree on their overlap $U_x \cap U_y$ by the unique continuation principle. Thus, for some neighborhood $U$ of $\p D$ in $\Omega$, there is a real-analytic function $v_0$ in $U$ so that 
\[
(\Delta + \lambda^2 + h_0) v_0 = 0 \text{ in $U$}, \qquad v_0|_{\p D} = u_0|_{\p D}, \ \ \p_{\nu} v_0|_{\p D} = \p_{\nu} u_0|_{\p D}.
\]
We may redefine $v_0 = u_0$ in $U \setminus \ol{D}$, so that $v_0 \in C^{1,1}(U)$ will satisfy 
\[
(\Delta + \lambda^2 + h_0) v_0 = 0 \text{ in $D \cap U$}, \qquad v_0|_{U \setminus \ol{D}} = u_0.
\]

By Lemma \ref{lemma_local_global} there are $v \in C^{1,1}(\Omega)$ and $h \in C^{\infty}(\ol{\Omega})$ that satisfy $v=v_0$ and $h=h_0$ near $\p D$, such that 
\[
(\Delta + \lambda^2 + h \chi_D) v = 0 \text{ in $\Omega$}, \qquad v_0|_{\Omega \setminus \ol{D}} = u_0.
\]
It remains to set $u_q = v$ in $D$ and $u_q = u_0$ in $\Omega \setminus \ol{D}$. Then $u_q \in H^2(\Omega)$ has the required properties.
\end{proof}

In the case of quadrature domains we instead use \eqref{quadrature_condition} to produce the required local solution.

\begin{proof}[Proof of Theorem \ref{thm_fb_example_quadrature}]
Let $G$ be the fundamental solution for $-\Delta$ in $\mR^n$, i.e.\ $G(x) = c_2 \log \,\abs{x}$ when $n=2$ and $G(x) = c_n \abs{x}^{2-n}$ when $n \geq 3$. Let $\mu$ be the distribution with $\supp(\mu) \subset D$ appearing in the definition of the quadrature domain $D$, and define 
\[
u = G \ast (\chi_D - \mu).
\]
Since $\chi_D - \mu$ is a compactly supported distribution, $u$ is a distribution in $\mR^n$ and it satisfies 
\[
\Delta u = \chi_D - \mu \text{ in $\mR^n$}.
\]
Moreover, if $x \in \mR^n \setminus \ol{D}$, we may take $H(y) = G(x-y)$ in \eqref{quadrature_condition} to obtain that 
\[
u|_{\mR^n \setminus \ol{D}} = 0.
\]
In particular, since $\supp(\mu) \subset D$, there is a neighborhood $U$ of $\p D$ in $\mR^n$ such that 
\begin{equation} \label{deltau_chid_condition}
\Delta u = \chi_D \text{ in $U$}, \qquad u|_{U \setminus \ol{D}} = 0.
\end{equation}
Note that $u \in C^{1,1}(U)$ using the $C^{1,1}$ regularity results for the no-sign obstacle problem \cite{AnderssonLindgrenShahgholian}.

Define $v_0 = u + u_0$ in $U$. We first claim that there is $h_0$ near $\p D$ with $\abs{h_0} \geq c > 0$ near $\p D$ so that 
\[
(\Delta + h_0 \chi_D) v_0 = 0 \text{ near $\p D$}.
\]
In fact, using the equations for $u$ and $u_0$, for any $h_0$ one has 
\[
(\Delta + h_0 \chi_D) v_0 = \chi_D + h_0 \chi_D v_0 \text{ in $U$}.
\]
This quantity vanishes near $\p D$ if we set 
\[
h_0 = -\frac{1}{v_0} \text{ near $\p D$}.
\]
The denominator is nonvanishing near $\p D$ since $u$ is continuous with $u|_{\p D} = 0$, and since $u_0$ is positive on $\p D$. We have thus found the required function $h_0$ near $\p D$.

By Lemma \ref{lemma_local_global} there are $h \in L^{\infty}(\Omega)$ and $v \in C^{1,1}_{\mathrm{loc}}(\Omega)$, with $h = h_0$ and $v = v_0$ near $\p D$, so that 
\[
(\Delta + h \chi_D) v = 0 \text{ in $\Omega$}, \qquad v|_{\Omega \setminus \ol{D}} = u_0|_{\Omega \setminus \ol{D}}.
\]
Setting $u_q = v$ in $\ol{D}$ and $u_q = u_0$ in $\Omega \setminus \ol{D}$ gives the required solution.
\end{proof}

\begin{Remark}
The only property of quadrature domains needed in the proof of Theorem \ref{thm_fb_example_quadrature} was the existence of a function $u$ satisfying \eqref{deltau_chid_condition}. The conclusion of Theorem \ref{thm_fb_example_quadrature}, also with a frequency $\lambda \geq 0$,  would hold for any domain $D$ that admits a function $u$ satisfying 
\[
(\Delta + \lambda^2) u = f \chi_D \text{ in $U$}, \qquad u|_{U \setminus \ol{D}} = 0,
\]
where $f$ is nonvanishing on $\p D$. In Theorem \ref{thm_fb_example_quadrature} we produced such a function as $u = G \ast (\chi_D - \mu)$, which can be understood as continuing the potential $G \ast \chi_D$ smoothly a little bit inside $D$ as the function $G \ast \mu$.
\end{Remark}

\section{Zero sets of Helmholtz solutions} \label{sec_helmholtz}

We complement the previous results by showing that for any bounded open set $D$ there is a solution $u_0$ of $(\Delta + \lambda^2) u_0 = 0$ in $\mR^n$ which is positive on $\p D$, under some restrictions on $D$ and $\lambda$. In the case $\lambda = 0$ one can take $u_0 \equiv 1$, so we will assume $\lambda > 0$. If one allows complex valued solutions the function $u_0 = e^{i\lambda x_1}$  is a nonvanishing solution in $\mR^n$. However, real valued solutions always have zeros. This is already seen in the case $n=1$ where any solution of $(\Delta+\lambda^2) u_0 = 0$ takes the form $u_0(x) = a \sin(\lambda x) + b \cos(\lambda x)$, and such a function has a zero in any closed interval of length $\pi/\lambda$. A similar result holds for Laplace eigenfunctions in compact manifolds (see the survey \cite{LogunovMalinnikova_survey}). The following version of this result shows that any real solution of $(\Delta + \lambda^2) u_0 = 0$ in $\mR^n$ has a zero in any ball of radius $\geq c_n/\lambda$.

\begin{Lemma} \label{lemma_helmholtz_local_zeros}
Let $\lambda > 0$ and $x^0 \in \mR^n$. There is a nonvanishing real function $u \in C^{\infty}(\ol{B(x^0,r)})$ solving $(\Delta + \lambda^2) u = 0$ in $B(x^0,r)$ if and only if $r < c_n/\lambda$, where $c_n$ denotes the first positive zero of the Bessel function $J_{\frac{n-2}{2}}$.
\end{Lemma}
\begin{proof}
By translation invariance we may assume $x^0 = 0$, and replacing $u(x)$ by $u(\lambda x)$ we may assume $\lambda = 1$. We consider radial solutions $v = v(r)$ of $(\Delta + 1) v = 0$ in $\mR^n$. Writing the Laplacian in polar coordinates, we see that $v$ should satisfy 
\[
v''(r) + \frac{n-1}{r} v'(r) + v(r) = 0 \text{ in $\mR_+$}.
\]
The substitution $v(r) = r^{\frac{2-n}{2}} w(r)$ leads to the Bessel equation 
\[
r^2 w''(r) + r w'(r) + (r^2 - (\frac{n-2}{2})^2) w(r) = 0  \text{ in $\mR_+$}.
\]
Since $v(r)$ should be bounded near $r=0$, one must have $v(r) = r^{\frac{2-n}{2}} J_{\frac{n-2}{2}}(r)$ (up to a scalar multiple). Thus there is a positive solution in $B(0,c_n)$.

For the converse, we argue by contradiction and suppose that $w \in C^{\infty}(\ol{B_r})$ is positive and solves $(\Delta + 1) w = 0$ in $B_r$ for some $r \geq c_n$. We now use the fact that if a Schr\"odinger equation has a nonvanishing solution, then it can be reduced to a divergence form equation with no zero order term. Writing $\gamma = w^2$, we see that 
\[
(\Delta + 1) u = 0 \text{ in $B_r$} \quad \Longleftrightarrow \quad \mathrm{div}(\gamma \nabla (\gamma^{-1/2} u)) = 0 \text{ in $B_r$}.
\]
Now since $r \geq c_n$, taking $u$ to be the radial solution $v$ above shows that $\gamma^{-1/2} u$ is positive in $B(0, c_n)$ but becomes zero on $\p B(0, c_n)$. By the maximum principle for the equation $\mathrm{div}(\gamma \nabla \,\cdot\,) = 0$, the maximum of $\gamma^{-1/2} u$ in $\ol{B(0, c_n)}$ should be attained at the boundary. This is a contradiction.
\end{proof}

We now turn to the question of determining if there is a real valued solution of $(\Delta + \lambda^2) u_0 = 0$ in $\mR^n$ that is nonvanishing on $\p D$. The following remark shows that this may be false when $\lambda > 0$ is an eigenvalue of $-\Delta$ in $D$.

\begin{Remark} \label{remark_nonvanishing_false}
Let $D = B(0,c_n)$ where $c_n$ is as in Lemma \ref{lemma_helmholtz_local_zeros}, and suppose $u_0$ is real and solves $(\Delta + \lambda^2) u_0 = 0$ in $\mR^n$. Let $v \in H^1_0(D)$ be the radial solution in the proof of Lemma \ref{lemma_helmholtz_local_zeros} with $v > 0$ in $D$. Since $v|_{\p D} = 0$, one has  
\[
\int_{\p D} u_0 \p_{\nu} v \,dS = \int_D (u_0 \Delta v - (\Delta u_0) v) \,dx = 0.
\]
However, $\p_{\nu} v < 0$ on $\p D$, which implies that $u_0$ must change sign on $\p D$. The same argument works for any sufficiently regular $D$ if $\lambda$ is the first Dirichlet eigenvalue and $\p D$ is connected.
\end{Remark}

The next result shows that if $\lambda$ is not a Dirichlet eigenvalue in $D$, there is a solution $u_0$ which is positive on $\p D$.

\begin{Proposition} \label{prop_u0_nonvanishing}
Let $D \subset \mR^n$ be a bounded Lipschitz domain if $n = 2,3$ (resp.\ $C^1$ domain if $n \geq 4$) so that $\mR^n \setminus \ol{D}$ is connected. Suppose that $\lambda > 0$ is not a Dirichlet eigenvalue of $-\Delta$ in $D$. Then there is a real valued $u_0$ solving $(\Delta + \lambda^2) u_0 = 0$ in $\mR^n$ so that $u_0$ is positive on $ \p D$.
\end{Proposition}

We will prove the above result in two steps. First we show that there is $v \in W^{1,p}(D)$ for some $p > n$ such that $(\Delta+\lambda^2)v=0$ in $D$ and $v|_{\p D}$ is positive. Then we apply a Runge approximation property, showing that $v$ can be approximated by functions $u|_{D}$ where $u$ solves $(\Delta + \lambda^2) u = 0$ in $\mR^n$. This kind of property is classical for second order elliptic equations and it follows from the unique continuation property. See \cites{Lax, Malgrange} for the original results in bounded domains and \cite{RulandSalo} for further references.

Approximation by Helmholtz solutions in $\mR^n$ has been used in scattering theory at least with respect to $L^2$ norms, see e.g.\ \cite{Weder_completeness} and references therein. We need a corresponding approximation result in the $C(\ol{D})$ norm. This is more involved than approximation in $L^2$ (for analytic functions this would correspond to Mergelyan's theorem instead of Runge's theorem). In order to achieve this we will assume some regularity on $D$ and work with Sobolev norms instead.

Below we say that $D$ is a $C^0$ domain if it is locally the region above the graph of a continuous function, and we define $H^{1,p}(D) = \{ u|_D \,:\, u \in W^{1,p}(\mR^n) \}$ with the quotient norm $\norm{v}_{H^{1,p}(D)} = \inf \{ \norm{u}_{W^{1,p}(\mR^n)} \,:\, u|_D = v \}$. The space $H^{1,p}(D)$, defined via restriction, coincides with the standard Sobolev space $W^{1,p}(D)$ whenever $D$ is a $W^{1,p}$ extension domain. The following version of the Runge approximation property will be relevant for us. In the rest of this section we allow functions to be complex valued.

\begin{Proposition} \label{prop_runge}
Let $1 < p < \infty$, let $\lambda > 0$, and let $D \subset \mR^n$ be a bounded $C^0$ domain such that $\mR^n \setminus \ol{D}$ is connected. Given any $v \in H^{1,p}(D)$ (possibly complex valued) with $(\Delta + \lambda^2) v = 0$ in $D$, there exist $u_j$ solving $(\Delta + \lambda^2) u_j = 0$ in $\mR^n$ so that 
\[
\norm{u_j - v}_{H^{1,p}(D)} \to 0 \text{ as $j \to \infty$.}
\]
 If $v$ is real valued, then so are $u_j$.
\end{Proposition}

The proof of Proposition \ref{prop_u0_nonvanishing} follows rather easily:

\begin{proof}[Proof of Proposition \ref{prop_u0_nonvanishing}]
Since $\lambda$ is not a Dirichlet eigenvalue in $D$, there is a real valued solution $v = 1 + w$ of $(\Delta + \lambda^2) v = 0$ in $D$ where $w \in H^1_0(D)$ is the unique solution of  
\[
(\Delta + \lambda^2) w = -\lambda^2 \text{ in $D$}.
\]
We claim that $v \in W^{1,p}(D)$ for some $p > n$. Note that the $W^{1,p}(D)$ and $H^{1,p}(D)$ norms are equivalent since Lipschitz domains are $W^{1,p}$ extension domains. Now, if  $v \in W^{1,p}(D)$ for some $p > n$, then by Proposition \ref{prop_runge} there are global solutions $u_j$ such that 
\[
\norm{u_j-1}_{C(\p D)} \leq \norm{u_j-v}_{C(\ol{D})} \leq C \norm{u_j-v}_{H^{1,p}(D)} \to 0
\]
where we used the Sobolev embedding. This shows the existence of a global solution that is positive on $\p D$.

To prove that $v \in W^{1,p}(D)$ for some $p > n$, we note that $w \in H^1_0(D)$ solves 
\begin{equation} \label{deltaw_eq}
\Delta w = -\lambda^2 v \text{ in $D$}.
\end{equation}
By Sobolev embedding the right hand side is in $L^{\frac{2n}{n-2}}(D)$ for $n \geq 3$, and in $L^r(D)$ for any $r < \infty$ for $n=2$. In particular the right hand side is in $W^{-1,r}(D)$ for any $r < \infty$ for $n=2,3,4$ and for $r = \frac{2n}{n-4}$ for $n \geq 5$. Since $D$ is Lipschitz, by \cite[Theorem 1.1]{JerisonKenig1995} one has $w \in W^{1,p}(D)$ for some $p > 3$ if $n=2,3$. This proves the claim for Lipschitz domains in dimensions $n=2,3$. For $C^1$ domains in dimensions $n \geq 4$, using \cite[Theorem 1.1]{JerisonKenig1995}, which holds with $p_0=1$ in $C^1$ domains, shows that $w \in W^{1,\frac{2n}{n-4}}(D)$. Returning to \eqref{deltaw_eq}, noting that the right hand side has more regularity, and iterating this argument shows that $v \in W^{1,p}(D)$ for all $p < \infty$ in the case of $C^1$ domains.
\end{proof}

To prove Proposition \ref{prop_runge}, it is convenient to introduce the operator 
\[
P(\lambda): C^{\infty}(S^{n-1}) \to C^{\infty}(\mR^n), \ \ P(\lambda)f(x) = \int_{S^{n-1}} e^{i\lambda x \cdot \omega} f(\omega) \,d\omega.
\]
Functions $u = P(\lambda) f$ are called Herglotz waves and they are particular solutions of $(\Delta + \lambda^2) u = 0$ in $\mR^n$. One can think of $f$ as a certain boundary value at infinity for $u$, and of $P(\lambda)$ as a Poisson integral that gives the solution of $(\Delta+\lambda^2) u = 0$ having boundary value $f$ at infinity.

The following proof, modelled after \cites{UhlmannVasy, PaivarintaSaloUhlmann}, will show that restrictions of Herglotz waves to $D$ are dense in the set of all Helmholtz solutions in $H^{1,p}(D)$.

\begin{proof}[Proof of Proposition \ref{prop_runge}]
By the Hahn-Banach theorem, it is enough to prove that any bounded linear functional $\ell$ on $H^{1,p}(D)$ that vanishes on $\{ P(\lambda)f|_{D} \,;\, f \in C^{\infty}(S^{n-1}) \}$ must also vanish on $\{ v \in H^{1,p}(D) \,;\, (\Delta + \lambda^2) v = 0 \text{ in $D$} \}$. We define a functional 
\[
\ell_1: W^{1,p}(\mR^n) \to \mC, \ \ \ell_1(u) = \ell(u|_D).
\]
Then $\ell_1$ is bounded on $W^{1,p}(\mR^n)$, and by duality 
\[
\ell_1(u) = (u, \mu)
\]
for some $\mu \in W^{-1,p'}(\mR^n)$, where $\frac{1}{p} + \frac{1}{p'} = 1$ and $(\,\cdot\,, \,\cdot\,)$ is the sesquilinear distributional pairing in $\mR^n$. Clearly $\mu = 0$ in $\mR^n \setminus \ol{D}$, so $\mu$ is a compactly supported distribution. Thus the condition $\ell(P(\lambda)f|_{D}) = 0$ for all $f \in C^{\infty}(S^{n-1})$ implies that 
\begin{equation} \label{orth1}
(P(\lambda)f, \mu) = 0 \text{ for all $f \in C^{\infty}(S^{n-1})$}.
\end{equation}

Let $G_{\lambda}$ be the outgoing fundamental solution of $\Delta + \lambda^2$, given by 
\[
G_{\lambda}(x) = c_{n,\lambda} \abs{x}^{-\frac{n-2}{2}} H_{\frac{n-2}{2}}^{(1)}(\lambda \abs{x})
\]
where $H_{\nu}^{(1)}$ is the Hankel function (see \cite[Section 1.2.3]{Yafaev}), and let $w$ be the distribution 
\[
w = G_{\lambda} \ast \mu.
\]
Then $w$ is a distributional solution of  
\[
(\Delta+\lambda^2) w = \mu \text{ in $\mR^n$}.
\]
By elliptic regularity $w \in W^{1,p'}_{\mathrm{loc}}(\mR^n)$ and $w$ is smooth outside $\ol{D}$, with the expression 
\begin{equation} \label{w_exterior}
w(x) = (G_{\lambda}(x-\,\cdot\,), \mu), \qquad x \in \mR^n \setminus \ol{D}. 
\end{equation}

Let $f \in C^{\infty}(S^{n-1})$ and $u = P(\lambda) f \in C^{\infty}(\mR^n)$. By \eqref{orth1} and the fact that $\mu$ has compact support, we have 
\begin{equation} \label{orth2}
0 = (u, \mu) = \lim_{r \to \infty} (u, \mu)_{B_r}
\end{equation}
where $(\,\cdot\,, \,\cdot\,)_{B_r}$ is the sesquilinear distributional pairing in $B_r$. We wish to use that $\mu = (\Delta + \lambda^2) w$. Since $w$ is not smooth in $\ol{D}$ we introduce a cutoff function $\chi \in C^{\infty}_c(\mR^n)$ with $0 \leq \chi \leq 1$ and $\chi = 1$ near $\ol{D}$. Writing $u = \chi u + (1-\chi)u$ and using that everything is smooth outside $\ol{D}$, we obtain from \eqref{orth2} that 
\begin{align*}
0 &= \lim_{r \to \infty}  \left( (\chi u, (\Delta + \lambda^2) w)_{B_r} + ((1-\chi) u, (\Delta + \lambda^2) w)_{B_r} \right] \\
 &= \lim_{r \to \infty} \left[ ((\Delta + \lambda^2) (\chi u), w)_{B_r} + ((\Delta + \lambda^2) ((1-\chi) u), w)_{B_r} + \int_{\p B_r} (u \ol{\p_{\nu} w} - (\p_{\nu} u) \ol{w}) \,dS \right].
\end{align*}
Since $(\Delta + \lambda^2) u = 0$, this reduces to 
\begin{equation} \label{orth3}
\lim_{r \to \infty} \int_{\p B_r} (u \ol{\p_{\nu} w} - (\p_{\nu} u) \ol{w}) \,dS = 0.
\end{equation}

Writing $x = r\theta$ where $r \geq 0$ and $\theta \in S^{n-1}$, the function $u = P(\lambda) f$ has the asymptotics 
\begin{align*}
u(r\theta) &= c_{n,\lambda}' r^{-\frac{n-1}{2}} \left[ e^{i\lambda r} f(\theta) + i^{n-1} e^{-i\lambda r} f(-\theta) \right] + O(r^{-\frac{n+1}{2}}), \\
\p_r u(r\theta) &= c_{n,\lambda}' r^{-\frac{n-1}{2}} i \lambda \left[ e^{i\lambda r} f(\theta) - i^{n-1} e^{-i\lambda r} f(-\theta) \right] + O(r^{-\frac{n+1}{2}}),
\end{align*}
as $r \to \infty$ (see \cite[Section 1.3]{Melrose}). We wish to study similar asymptotics for the outgoing function $w$. Using \eqref{w_exterior} and asymptotics for the Hankel function, we see that (as in \cite[Section 1.2.3]{Yafaev})
\begin{align*}
w(r\theta) &= c_{n,\lambda}'' r^{-\frac{n-1}{2}} e^{i\lambda r} \hat{\mu}(\lambda \theta) + O(r^{-\frac{n+1}{2}}), \\
\p_r w(r\theta) &= c_{n,\lambda}'' r^{-\frac{n-1}{2}} i \lambda e^{i\lambda r} \hat{\mu}(\lambda \theta) + O(r^{-\frac{n+1}{2}})
\end{align*}
where $\hat{\mu} \in C^{\infty}(\mR^n)$ is the Fourier transform of the compactly supported distribution $\mu$. Above $c_{n,\lambda}$, $c_{n,\lambda}'$ and $c_{n,\lambda}''$ are nonzero constants.

Inserting the asymptotics for $u$ and $w$ into \eqref{orth3} and noting that the terms containing $f(-\theta)$ cancel yields that 
\[
\int_{S^{n-1}} f(\theta) \ol{\hat{\mu}(\lambda \theta)} \,d\theta = 0.
\]
Since this is true for all $f \in C^{\infty}(S^{n-1})$ we must have $\hat{\mu}(\lambda \theta) = 0$. In particular, $w$ has the asymptotics 
\[
w(r \theta) = O(r^{-\frac{n+1}{2}}), \qquad \p_r w(r\theta) = O(r^{-\frac{n+1}{2}}).
\]
Since $w$ is outgoing and satisfies $(\Delta + \lambda^2) w = 0$ in $\mR^n \setminus \ol{D}$, the Rellich uniqueness theorem (see e.g.\ \cite{Hormander1973}) implies that $w = 0$ outside a large ball. Since $\mR^n \setminus \ol{D}$ is connected, the unique continuation principle gives that $w = 0$ in $\mR^n \setminus \ol{D}$.

We have now proved that the condition \eqref{orth1} implies that 
\[
\mu = (\Delta + \lambda^2) w \text{ in $\mR^n$}
\]
for some $w \in W^{1,p'}(\mR^n)$ vanishing in $\mR^n \setminus \ol{D}$. Now let $v \in H^{1,p}(D)$ be any solution of $(\Delta+\lambda^2) v = 0$ in $D$, and let $\tilde{v} \in W^{1,p}(\mR^n)$ be an extension of $v$. Then we have 
\[
\ell(v) = \ell_1(\tilde{v}|_D) = (\tilde{v}, \mu) = (\tilde{v}, (\Delta + \lambda^2) w) = ((\Delta+\lambda^2) \tilde{v}, w).
\]
Since $D$ is a bounded $C^0$ domain and since $w \in W^{1,p'}(\mR^n)$ vanishes in $\mR^n \setminus \ol{D}$, there are $w_j \in C^{\infty}_c(D)$ with $w_j \to w$ in $W^{1,p'}(\mR^n)$ (this is proved as in \cite[Theorem 3.29]{McLean}). It follows that 
\[
\ell(v) = \lim_{j \to \infty} ((\Delta+\lambda^2) \tilde{v}, w_j) = 0
\]
since $(\Delta+\lambda^2) \tilde{v} = 0$ in $D$. This concludes the proof.
\end{proof}

\section{Free boundary methods} \label{sec_fb} 

By arguments from Section  \ref{sec:fb-connection} we know that $u = u_q - u_0$  satisfies the equation (see \eqref{fbp_eq4}--\eqref{fbp_eq5})
\begin{equation}\label{eq:fb-no-sign}
\Delta u = f(x) \chi_{\{u \neq 0\}} \qquad \hbox{ in $B(x^0, r)$},
\end{equation}
where we may assume that $f(x) > 0 $ in some neigborhood of $x^0 \in  \partial {\{u \neq 0\}} $.
The above equation has been treated extensively in the literature, and all regularity aspects of the problem are resolved and sorted out; see e.g.\  \cite{AnderssonLindgrenShahgholian} and the references therein. Theorem \ref{Thm:FB-reg} follows directly from \cite[Theorem 1.3]{AnderssonLindgrenShahgholian}, and the proof of Theorem \ref{Thm:FB-reg1} is sketched below. Theorem \ref{Thm:FB-reg-dichotomy} in the case where $h$ is Dini or Lipschitz continuous follows from these results, and the higher regularity results follow from the method of \cite{KN} (see also \cite[Section 6.4]{PetrosyanShahgholianUraltseva}).

We shall now give classical examples of singularities that can appear in the obstacle problem.

\begin{ExampleNoNumber}
 (\cite{KN}, \cite{Sc}, \cite{Sa1}, \cite{Sa2})
We recall from \cite[page 387--390]{KN} an explicit example  of cusps appearing in the free boundary.
 These cusps  are represented by the curves
$$x_2=\pm x_1^{\mu/2}, \qquad 0\leq x_1 \leq 1,$$
where  $\mu = 4k +1,$ ($k=1,2, \cdots$) gives
non-negative solutions and
 $\mu = 4k +3,$ ($k=0,1, \cdots$) gives solutions
that become negative on the negative $x_1$-axis and near the origin.
The solution is defined locally by
$$u(x)=x_2^2  - \frac{2}{1 + \mu/2} \rho^{1+ \mu/2 }
\sin(1 + \mu/2)\theta + \cdots, \qquad x \in \Omega, \
|x| < \epsilon,
$$
for $\epsilon$ small.
Here  we have used  both
 real and complex
  notation
$$x=(x_1,x_2), \qquad
z=\rho e^{i\theta}, \quad 0\leq \theta \leq 2\pi.
$$
Also the domain $\Omega$ is the image of the  set
$$\{z: \  |z| <1 , \ \hbox{Im }z >0\} $$
under the conformal mapping $f(z)=z^2 + iz^\mu$.

\end{ExampleNoNumber}

\begin{proof}[Proof of  Theorem \ref{Thm:FB-reg1}]
The proof of Theorem  \ref{Thm:FB-reg1} when $h \in C^{0,1}(B(x^0, r))$ (i.e.\ case (1)) is somehow hidden in  \cite{Caff-Sha2004} (see their proof of Main Theorem), where it is proven that the singular set  of the free boundary lies in a $C^1$-manifold. In particular this means that whenever we blow up a solution at a singular free boundary point through any sequence $u(r_j x + x^0) /r_j^2$  (here $u = u_q - u_0$ satisfies \eqref{eq:fb-no-sign}) it will converge to a fixed  polynomial $p(x)$, with the the free boundary $ \{p= \nabla p =0\}$, and regardless of the sequence $\{r_j\}$.
This in particular implies that the limiting free boundary lies in a plane (or lower dimensional plane) which after  translation and rotation  we assume it is  $\{x_n =0\}$. From here it follows that the free boundary approaches this plane tangentially, whence the statement  b) follows whenever the free boundary has a  cusp at $x^0$.

 In case $x^0$ is not a cusp  point, then by Theorem \ref{Thm:FB-reg} in a vicinity of $x^0$  the free boundary is $C^1$, 
 and $u \geq 0$, and obviously $\p_e u \geq 0$ in a smaller neighbourhood of $x^0$. By results of \cite{Allen-Sha} (see section 1.4.2) the free boundary is $C^{1,\alpha}$, and Theorem  \ref{Thm:FB-reg1}  is proved in case (1).

To prove Theorem  \ref{Thm:FB-reg1} in the case (2) we work with $v= \partial_e u$, where by the assumption 
$e = \nabla u_0 (x^0) \neq {\bf 0}$. Since $u_q = u_0$ in $D^c$, we also have $\nabla u_q (x^0) \neq {\bf 0}$. 
As before with $u = u_q - u_0$ we have $-(\Delta+\lambda^2) \partial_e u = \partial_e(hu_q)  \chi_D = (\partial_eh u_q + h \partial_eu_q )  \chi_D$ close to $x^0$.
Since $u_0= u_q=0$ on $\partial D \cap B_r (x^0)$ and $\partial_e u_q (x^0) >0 $ and $h(x^0) \neq 0 $, we have that $\partial_e u$ satisfies the hypothesis of Theorem \ref{Thm:FB-reg1} case (1), and hence the result follows.
\end{proof}

\end{document}